\documentclass[reqno]{amsart}

\usepackage{amsmath}
\usepackage{amssymb}

\theoremstyle{plain}
\newtheorem{theorem}{Theorem}[section]
\newtheorem{lemma}[theorem]{Lemma}
\newtheorem{corollary}[theorem]{Corollary}
\newtheorem{proposition}[theorem]{Proposition}

\theoremstyle{definition}
\newtheorem{remark}[theorem]{Remark}

\newtheorem{example}[theorem]{Example}

\numberwithin{equation}{section}

\newcommand{\BC}{{\mathbb C}}

\newcommand{\BR}{{\mathbb R}}

\newcommand{\cG}{{\mathcal G}}

\newcommand{\cX}{{\mathcal X}}
\newcommand{\cZ}{{\mathcal Z}}

\newcommand{\wtilA}{\widetilde{A}}\newcommand{\wtilB}{\widetilde{B}}

\newcommand{\whatA}{\widehat{A}}\newcommand{\whatB}{\widehat{B}}

\newcommand{\whatE}{\widehat{E}}

\newcommand{\al}{\alpha}

\newcommand{\Ga}{\Gamma}
\newcommand{\De}{\Delta}

\newcommand{\la}{\lambda}\newcommand{\La}{\Lambda}

\newcommand{\Si}{\Sigma}
\newcommand{\Up}{\Upsilon}

\newcommand{\om}{\omega}
\newcommand{\rank}{\textup{rank\,}}
\newcommand{\ran}{\textup{ran\,}}

\newcommand{\im}{\textup{Im}}
\newcommand{\re}{\textup{Re}}
\newcommand{\kr}{\textup{Ker\,}}
\newcommand{\diag}{\textup{diag\,}}

\newcommand{\mat}[2]{\ensuremath{\left[\begin{array}{#1}#2\end{array} \right]}}

\newcommand{\sbm}[1]{\left[\begin{smallmatrix} #1\end{smallmatrix}\right]}

\newcommand{\tu}[1]{\textup{#1}}
\newcommand{\wtil}[1]{{\widetilde{#1}}}
\newcommand{\what}[1]{{\widehat{#1}}}

\newcommand{\half}{\frac{1}{2}}
\newcommand{\ands}{\quad\mbox{and}\quad}

\newcommand{\PRO}{\mathcal{PRO}}

\begin{document}


\title[Positive real odd rational matrix functions]{The convex invertible cone structure of positive real odd rational matrix functions}

\author[S. ter Horst]{S. ter Horst}
\address{S. ter Horst, Department of Mathematics, Research Focus Area:\ Pure and Applied Analytics, North-West
University, Potchefstroom, 2531 South Africa and DSI-NRF Centre of Excellence in Mathematical and Statistical Sciences (CoE-MaSS)}
\email{Sanne.TerHorst@nwu.ac.za}

\author[A. Naude]{A. Naud\'{e}}
\address{A. Naude, Department of Mathematics, Research Focus Area:\ Pure and Applied Analytics, North-West
University, Potchefstroom, 2531 South Africa and DSI-NRF Centre of Excellence in Mathematical and Statistical Sciences (CoE-MaSS)}
\email{naudealma@gmail.com}

\thanks{This work is based on the research supported in part by the National Research Foundation of South Africa (Grant Number 90670 and 118513).
}




\begin{abstract}
Positive real odd matrix functions, often referred to as positive real lossless matrix functions, play an important role in many applications in multi-port electrical systems. In this paper we present closer analogues to some of the known results for the scalar, one-port, case in the multi-port setting. Specifically, we determine necessary and sufficient conditions for the well studied partial fraction formula to represent functions in the class of positive real odd matrix functions, and explicit minimal state space realization formulas for the inverse (admittance) of a function in this class, which itself is also a positive real odd matrix function. Doing so, enables us to provide a partial analogue of the pole-zero interlacing behavior from the scalar case.
\end{abstract}

\maketitle

\section{Introduction}

The research on positive real odd functions ($\PRO$ for short), often also called positive real lossless functions, got spearheaded by the pioneering work in electrical engineering of Foster \cite{F24}, Cauer \cite{C26} and Brunce \cite{B31}, once it was observed by Foster that this class of functions appears as the impedances (and their admittances) of lumped one-port electrical circuits generated by inductances and capacitors; see also \cite{N66,B68,AV73}. One of the main results of Foster \cite{F24} is the seminal canonical form for one-port reactance functions, namely $f$ is in $\PRO$ if and only if it has the form
\begin{equation}\label{FosterScalar}
f(z)= a_0 z + \sum_{k=1}^s \frac{a_k z}{z^2+\om_k^2},\quad a_0\geq 0,\  a_k,\om_k\geq 0,\ k=1,\ldots,s.
\end{equation}
In words, all poles are on $i\BR\cup \{\infty\}$, simple, come in complex conjugate pairs (apart from $0$ and $\infty$) and have positive residues. This implies that the zeros of $f$ interlace the poles on $i\BR$, and, as a consequence, the involution (or admittances) $1/f$ is also in $\PRO$. In particular, $\PRO$ is a convex invertible cone \cite{CL07}, i.e., a convex cone which is closed under inversion. Convex cones play an important role in many parts of applied mathematics; the concept of convex invertible cones in system and control theory was propagated by Cohen and Lewkowicz \cite{CL09,CL07,CL97a,CL97b}.

Matrix-valued $\PRO$ functions appear when multi-port electrical systems built from inductances, capacitors and gyrators are considered, and they have been studied intensely for many decades, leading to a vast literature on this topic, cf., \cite{N66,AV73} for references and a discussion of the classical work and \cite{BR14,CT08,R10} for some more recent results. In this setting, for an integer $m\geq 0$ we write $\PRO_m$ for the class of $m\times m$ rational matrix functions $F$ so that
\begin{equation}\label{PROm}
\begin{aligned}
&\re(F(z)) \geq 0, \mbox{ for $\re(z)>0$},\quad F(t)\in\BR^{m\times m}  \mbox{ for $t\in\BR$},\\
&\qquad-F(z)=F(-\overline{z})^* \mbox{ for $z$ not a pole of $F$}.
\end{aligned}
\end{equation}
Here for any square matrix $K$, with $K\geq 0$ ($K\leq 0$) we indicate that $K$ is positive (negative) semidefinite, while for matrices $K,L$, $K\geq L$ should be interpreted as $K-L\geq 0$. The last condition in \eqref{PROm}, stating the $F$ is odd on $\BR$, is equivalent to $\re (F(z))=0$ for $z\in i\BR$, not a pole of $F$, which gives the connection with lossless systems. It is easy to prove from the defining conditions \eqref{PROm} that $\PRO_m$ is also a convex invertible cone, with invertibility in the form of involution, i.e., $F^{-1}(z):=F(z)^{-1}$ in case $\det F(z)\not\equiv 0$. The matrix form of the partial fraction expansion \eqref{FosterScalar} has also been studied extensively, cf., \cite{N66,AV73}, and takes the form
\begin{equation} \label{FosterMat1}
F(z) = z Q + R + \sum_{j=1}^s \frac{1}{z^2 + \om_j^2} \left(zQ_j + R_j \right),
\end{equation}
where $\om_j\geq0$, $Q,R,Q_j,R_j\in \BR^{m\times m}$ with $Q,Q_j\geq 0$ and $R,R_j$ skew-symmetric. However, not all functions $F$ of this form are in $\PRO_m$, and we have not been able to find in the literature precise conditions on the parameters in \eqref{FosterMat1} which guarantee that $F$ is in $\PRO_m$. In Theorem \ref{T:Main2} below we identify the remaining condition to be
\begin{equation}\label{FosterCon1}
-\om_j Q_j \leq i R_j \leq \om_j Q_j,\quad j=1,\ldots,s.
\end{equation}
Although various successful approaches to the positive real lossless synthesis problems have been obtained, see Remark \ref{R:Synth} below, these do not seem to rely on a condition of the form \eqref{FosterMat1}, but rather seem to use the fact that $F$ is in $\PRO_m$ directly. Using minimal state space realization formulas for the class $\PRO_m$ collected in Section \ref{S:Transfer}, we show that adding \eqref{FosterCon1} provides necessary and sufficient conditions for \eqref{FosterMat1} to be a characterization of $\PRO_m$. In fact, we provide a concrete construction of a minimal realization, satisfying the appropriate conditions, for a function $F$ of the form \eqref{FosterMat1} satisfying \eqref{FosterCon1}.

Note that a $m\times m$ rational matrix function $F$ is in $\PRO_m$ if and only if $N(z):= iF(-iz)$ is a Nevanlinna function. Annemarie L\"{u}ger \cite{L20} confirmed to us that the conditions \eqref{FosterCon1} can also be derived from the integral representations that exist for the class of matrix-valued Nevanlinna functions. However, condition \eqref{FosterCon1} does not seem to have appeared in the literature before, and the explicit construction of the state space realization based on this condition also seems to be new.

The main part of the paper, however, involves the convex invertible cone structure of $\PRO_m$, more specifically, the fact that $\PRO_m$ is closed under inversion. This is not difficult to prove from \eqref{PROm}, however, in Section \ref{S:Inverse} we present explicit formulas for minimal state space realizations of $F^{-1}$, of the types presented in Section \ref{S:Transfer}, based on given state space realizations for $F$; see Theorems \ref{T:inv1} and \ref{T:inv2} below. One of the advantages of this direct approach is that it enables us to analyse the pole-zero structure of functions in $\PRO_m$ by comparing eigenvalues of the state matrices of the realizations of $F$ and $F^{-1}$. This leads to a partial analogue of the pole-zero interlacing property in the scalar case, which is presented in the final section, see Theorem \ref{T:Interlace}. In particular, we obtain that between two subsequent poles of $F\in\PRO_m$ on $i\BR$ zeros can occur with multiplicities that add up to at most $m$, and likewise between two zeros. Different from the scalar case, however, independent of zeros (poles) occurring between two subsequent poles (zeros) it can also happen that a zero (pole) occurs at one or both of the two poles (zeros), as poles and zeros can occur at the same point.

Finally, we point out that some of the results that we derive here have been presented in the proceedings paper \cite{tHN20}, without proof, except for an alternative, less constructive proof of the sufficiency part of Theorem \ref{T:Main2}.

\section{Preliminaries about transfer function representations}\label{S:TransferPre}

For the readers convenience we recall here some basic result about transfer function representations for real rational matrix functions, that will be used throughout the paper.  Here a real rational matrix function is a matrix function whose entries are ratios of two real polynomials, although we will consider them as functions acting on $\BC$. For proofs, further results and background on this topic we refer to \cite{ZDG96,DP00,D89,KM06}.

Recall that an $m\times k$ (real) rational matrix function $F$ is called proper in case $\lim_{z\to\infty} F(z)$ exists. In case $F$ is proper, there exist matrices $A\in \BR^{n \times n}$, $B\in\BR^{n\times k}$, $C\in\BR^{m\times n}$ and $D\in\BR^{m\times k}$, for some positive integer $n$, so that
\begin{equation}\label{realize}
F(z)=D+C(zI-A)^{-1}B,\quad \mbox{for $z\in\BC$ not a pole of $F$.}
\end{equation}
Such a representation of $F$ is called a {\em transfer function representation}, since the right hand side of \eqref{realize} coincides with the transfer function of the linear state space system
\begin{equation}\label{linsys}
\Si:\ \ \left\{
\begin{array}{l}
  \dot{x}(t) = A x(t)+ B u(t),\  \ x(0)=0, \\
  y(t) = C x(t) + D u(t),\ \  t\geq 0.
\end{array}
 \right.
\end{equation}
In this context $n$ is called the state space dimension. Note that $D=\lim_{z\to\infty} F(z)$.

The function $F$ has many different transfer function representations \eqref{realize}. However, if we demand that the state space dimension $n$ is as small as possible, then the representation \eqref{realize} is unique up to transformations of the state space. In such a case we say that the transfer function representation \eqref{realize} is {\em minimal}. To test for minimality, define the controllability gramian $\cG_\tu{c}$ and observability gramian $\cG_\tu{o}$ as
\[
\cG_\tu{c}:=\sum_{j=0}^{n-1} A^j BB^* A^{*j} \ands
\cG_\tu{o}:=\sum_{j=0}^{n-1} A^{*j} C^*C A^{j}.
\]
Then the system $\Si$ in \eqref{linsys}, or the pair $(A,B)$, is called {\em controllable} if $\cG_\tu{c}$ is positive definite, while $\Si$, or the pair $(C,A)$, is called {\em observable} if $\cG_\tu{o}$ is positive definite. Note that the pair $(A,B)$ is controllable if and only if $(B^*,A^*)$ is an observable pair. It turns out that \eqref{realize} is a minimal transfer function realization precisely when $\Si$ is controllable and observable.

Whenever $F$ is not proper, it is possible to write $F(z)=F_\tu{p}(z)+ P(z)$ with $F_\tu{p}$ a proper rational matrix function and $P$ a matrix polynomial. Then $F_\tu{p}$ does admit a transfer function representation of the form \eqref{realize}. However, it is possible to write even a non-proper $F$ as the transfer function of a linear state space system, if one also allows descriptor systems, also referred to as singular systems or differential-algebraic systems, cf., \cite{D89,KM06}. A {\em descriptor system} is a linear state space system of the form
\begin{equation}\label{lindessys}
\Si_\tu{descr}:\ \ \left\{
\begin{array}{l}
  E\dot{x}(t) = A x(t)+ B u(t),\  \ x(0)=0, \\
  y(t) = C x(t) + D u(t),\ \  t\geq 0,
\end{array}
 \right.
\end{equation}
with $A$, $B$, $C$ and $D$ as before and $E\in \BR^{n \times n}$. In fact, the matrices $A$ and $E$ need not be square, but in this paper we will only encounter the square case. The descriptor system $\Si_\tu{descr}$, or the pair $(E,A)$, is called {\em regular} in case $\det (zE-A)\not\equiv 0$. The transfer function of a regular descriptor system $\Si_\tu{descr}$ is given by
\begin{equation}\label{realdes}
F(z)=D+C(zE-A)^{-1}B,\quad \mbox{for $z$ with $\det (zE-A)\neq 0$,}
\end{equation}
and any real rational matrix function $F$ appears as the transfer function of a regular descriptor system. Also here, a transfer function representation \eqref{realdes} is called {\em minimal} whenever the state space dimension $n$ is as small as possible. It is less straightforward to test minimality of a descriptor transfer function representation. In \cite[Theorem 6.2]{FJ04} a necessary and sufficient condition based on Hautus tests criteria is presented. Concretely, the representation \eqref{realdes} is minimal if and only if the folowing five conditions are met
\begin{align*}
\mbox{(i) } \rank \mat{cc}{z E - A & B}=n\mbox{ for all }z\in\BC,\quad
& \mbox{ (ii) } \rank \mat{cc}{E & B}=n,\\
\mbox{(iii) } \rank \mat{cc}{z E^T - A^T & C^T}=n \mbox{ for all }z\in\BC, \quad
& \mbox{ (iv) } \rank \mat{cc}{E^T & C^T}=n,\\
\mbox{(v) } A (\kr E)\subset \im\, E.
\end{align*}
Finally, after a transformation of the state space, it is always possible to write a regular descriptor system in its so called {\em Weierstrass form}. In this form the state space decomposes as an orthogonal direct sum $\BR^{n}=\BR^{n_1}\oplus \BR^{n_2}$ in such a way that with respect to this decomposition the matrices $E$, $A$, $B$ and $C$ take the form
\[
E=\mat{cc}{I_{n_1}&0\\ 0 & N},\quad A=\mat{cc}{A_1 & 0\\ 0& I_{n_2}},\quad B=\mat{c}{B_1\\B_2},\quad C=\mat{cc}{C_1 & C_2},
\]
where $N$ is a nilpotent matrix.
%
%
%
%
%

\section{Transfer function representations of $\PRO_m$ functions}\label{S:Transfer}

The main result in this section (Theorem \ref{T:Transfer}) appears to be well known, and is included mainly because it is required for our further analysis of $\PRO_m$ functions in the remainder of this paper. We could not find the precise statement in the literature, hence, for completeness, we indicated below how it can be obtained from some known results in e.g.\ \cite{R10}.

\begin{theorem}\label{T:Transfer}
An $m\times m$ rational matrix function $F$ is in $\PRO_m$ if and only if it admits a realization of the form
\begin{equation}\label{RealFormPRO}
F(z)=z M +D + B^T (z I_n -A)^{-1}B,
\end{equation}
for some integer $n\geq 0$, $M,D\in\BR^{m \times m}$, $B\in\BR^{n \times m}$ and $A\in\BR^{n \times n}$ with
\begin{equation}\label{RealFormPROcon}
M\geq 0,\quad A^T=-A,\quad D^T=-D,\quad \mbox{$(A,B)$ a controllable pair.}
\end{equation}
\end{theorem}


\begin{proof}[\bf Proof] The transfer function characterization of $\PRO_m$ via \eqref{RealFormPRO} with conditions \eqref{RealFormPROcon} follows from \cite{AV73}, see also \cite{R10}. Indeed, by Proposition 7 in \cite{R10} $F$ is positive real (first two conditions in \eqref{PROm}) if and only if its strictly polynomial part is of the form $zM$ with $M\geq 0$ and its proper part is also positive real. It is then clear that $F$ in $\PRO_m$ is equivalent to $F(z)=zM +F_0(z)$ with $F_0$ in $\PRO_m$ and proper. For the proper part $F_0$ one can apply the Positive Real Lemma for proper, positive real, lossless functions \cite[Theorem 8]{R10}, applying a state space similarity in case the solution $X$ to the Lur'e equations ((5) in \cite{R10} with $K=0$ and $J=0$) is not equal to the identity matrix. Recall here that lossless is a different terminology for the odd-property in $\PRO_m$.
\end{proof}


From the previous theorem, we easily get a descriptor characterization in Weierstrass form.

\begin{theorem}\label{T:Weierstrass}
A $m \times m$ rational matrix function $F$ is in $\PRO_m$ if and only if it admits a minimal descriptor realization of the form
\begin{equation}\label{Weierstrass1}
F(z)=D^\circ+C^{\circ T}(z E^\circ-A^\circ)^{-1}B^\circ,
\end{equation}
where we set $q=\rank M$ and factor $M=K^TK$ with $K\in\BR^{q \times m}$, and
\begin{equation}\label{Weierstrass2}
\begin{aligned}
&A^\circ=\mat{ccc}{A&0&0\\ 0&I_q&0\\0&0&I_q} ,\quad  E^\circ=\mat{ccc}{I_n&0&0\\0&0&I_q\\0&0&0},\\
&  B^\circ\!=\!\mat{c}{B\\ 0\\ -K},\quad  C^\circ=\mat{c}{B\\ K\\ 0},\quad D^\circ=D,
\end{aligned}
\end{equation}
with $M,D\in\BR^{m \times m}$, $B\in\BR^{n \times m}$ and $A\in\BR^{n \times n}$ matrices satisfying \eqref{RealFormPROcon}.
\end{theorem}

The proof follows by direct computation and is left to the interested reader. Again, we include this result as it plays an important role in the sequel.

Since the state matrix $A$ in \eqref{RealFormPRO} is skew-symmetric, it is clear that no Jordan blocks of size larger than one can appear, which is also expressed in the form of the Foster representation. Hence, it makes more sense to define the multiplicity of a pole $\om\neq \infty$ of a function $F\in\PRO_m$ to be the dimension of the eigenspace of $\om$ as an eigenvalue of the state matrix $A$ in the minimal realization of Theorem \ref{T:Transfer}, while the multiplicity of $\infty$ as a pole of $F$ is defined as $\rank M$. In this way, the multiplicities of the finite poles add up to the McMillan degree of the proper part of $F$, i.e., to the minimal state space dimension. The zeros of $F$ are then defined to be the poles of $F^{-1}$ in case $\det F(z)\not\equiv 0$, and their multiplicities are the multiplicities of the corresponding poles of $F^{-1}$.

Next we show that the multiplicities of the poles cannot exceed $m$.

\begin{corollary}\label{C:EigenBound}
For $F\in\PRO_m$ every pole on $i\BR$, $\infty$ included, has a multiplicity of at most $m$.
\end{corollary}

\begin{proof}[\bf Proof]
Following Theorem \ref{T:Transfer}, let $\omega_1, ... ,\omega_s$ be the non-zero eigenvalues of $A$ on $i\mathbb{R}_+$ with multiplicities $k_1, ... ,k_s.$ There exists an orthogonal matrix $U$ so that
\begin{equation*}
    U^TAU=\diag\left(A_1, \hdots ,A_s,0\right), \quad \text{with } A_j = \omega_j\begin{bmatrix} 0 & I_{k_j} \\ -I_{k_j} & 0 \end{bmatrix} \text{ and } \sum_{j=1}^sk_j=k,
\end{equation*}
with the $0$ in the last block diagonal entry indicating a block zero matrix of size $(n-2k)\times (n-2k)$. Now decompose $U^TB$ accordingly as\begin{equation*}
    U^TB = \begin{bmatrix}B_1 \\ \vdots \\ B_{s} \\ \widetilde{B} \end{bmatrix}, \quad \text{with} \quad B_j \in \mathbb{R}^{2k_j \times m} \text{ for } 1 \le j \le s \text{ and } \widetilde{B} \in \mathbb{R}^{n-2k \times m}.
\end{equation*} Since $(A,B)$ is a controllable pair, it follows that \begin{equation*} \begin{aligned}
    n&=\rank U^T\begin{bmatrix} B & AB & \hdots & A^{n-1}B \end{bmatrix}\\
    &= \rank \begin{bmatrix}U^TB & U^TAUU^TB & \hdots & \left(U^TAU\right)^{n-1}U^TB \end{bmatrix} \\
    &=\rank \begin{bmatrix} B_1 & A_1B_1 & \hdots & A_1^{n-1}B_1 \\ \vdots & \vdots & \ddots &\vdots \\ B_s & A_sB_s & \hdots & A_s^{n-1}B_s \\ \widetilde{B} & 0 & \hdots & 0 \end{bmatrix}, \end{aligned}
\end{equation*}
which is true only if $\rank \widetilde{B} = n-2k$. Thus $n-2k \le m.$ This proves that the multiplicity of $0$ as a pole of $F$ is at most $m.$
Again from the controllability of the pair $(A,B)$, it also follows for any $1 \le j \le s$ that
\begin{equation*}
2k_j = \rank \begin{bmatrix}B_j & A_jB_j & \hdots & A_j^{n-1}B_j \end{bmatrix}.
\end{equation*}
Since $A_j^2 =-\omega_j^2I_{2k_j}$, it follows for $n=2(r+1)$ that
\begin{equation*}
2k_j=\rank \begin{bmatrix} B_j & A_jB_j & -\omega_j^2B_j&\hdots & \left(-\omega_j^2 \right)^rA_jB_j \end{bmatrix}
= \rank \begin{bmatrix}B_j & A_jB_j \end{bmatrix}
\end{equation*}
and for $n=2r+1$ that
\begin{equation*}
2k_j=\rank \begin{bmatrix} B_j & A_jB_j & -\omega_j^2B_j&\hdots & \left(-\omega_j^2 \right)^rB_j \end{bmatrix}= \rank \begin{bmatrix}B_j & A_jB_j \end{bmatrix}.
\end{equation*}
Thus $2k_j=\rank\begin{bmatrix} B_j & A_jB_j \end{bmatrix}$ and from $\begin{bmatrix} B_j & A_jB_j \end{bmatrix} \in \mathbb{R}^{2k_j \times 2m}$ it follows that $k_j \le m$ for every $1 \le j \le s.$
Lastly, the multiplicity of $\infty$ as a pole of $F$ is given by $\rank M\leq m$.
\end{proof}

Since zeros are poles of $F^{-1}$, the next corollary follows immediately.

\begin{corollary}\label{C:EigenBound2}
For $F\in\PRO_m$ every zero on $i\BR$, $\infty$ included, has a multiplicity of at most $m$.
\end{corollary}

\section{The Foster representation}\label{S:Foster}

In this section we prove the Foster representation formula for functions in $\PRO_m$, that is, we prove the following theorem.

\begin{theorem}\label{T:Main2}
An $m\times m$ rational matrix function $F$ is in $\PRO_m$ if and only if $F$ is of the form
\begin{equation} \label{FosterMat2}
F(z) = z Q + R + \sum_{j=1}^s \frac{1}{z^2 + \om_j^2} \left(zQ_j + R_j \right),
\end{equation}
where $\om_j\geq 0$, $Q,R,Q_j,R_j\in \BR^{m\times m}$ with $Q,Q_j\geq 0$ and $R,R_j$ skew-symmetric so that
\begin{equation}\label{FosterCon2}
-\om_j Q_j \leq i R_j \leq \om_j Q_j,\quad j=1,\ldots,s.
\end{equation}
\end{theorem}

\begin{remark}\label{R:Synth}
The observation that functions in $\PRO_m$ admit a Foster representation \eqref{FosterMat2}, as the natural analogue of the scalar representation \eqref{FosterScalar}, already appears in many classical texts, e.g., Chapter 7 in \cite{N66} and Chapter 9  in \cite{AV73}, but without the precise condition \ref{FosterCon2} required for the reverse direction claim. We have also not encountered condition \ref{FosterCon2} in more recent papers on positive real (odd or lossless) functions, e.g., \cite{BR14,CF09,XL04,ZLX02,B11,RRV15}. In both \cite{N66,AV73} significant attention is given to the reverse direction, in the context of the impedance synthesis problem, but via different approaches. In \cite[pp.\ 206--212]{N66} an intricate recursive procedure is applied, while in \cite[Section 9.3]{AV73} it is used that any minimal realization of a function in $\PRO_m$ must satisfy a KYP equality from the corresponding bounded real lemma. To the best of our knowledge,  condition \eqref{FosterCon2} has not appeared in the literature before (apart from the proceeding paper \cite{tHN20} where we announced it).
\end{remark}

\begin{proof}[\bf Proof of necessity part of Theorem \ref{T:Main2}]
We first proof the necessity of \eqref{FosterMat2} and \eqref{FosterCon2}. Let $F\in\PRO_m$. Then $F$ admits a representation as in \eqref{RealFormPRO} with $A,B,M,D$ real matrices satisfying \eqref{RealFormPROcon}. We may assume $n$ is even, at the expense of loosing controllability. Indeed, if $n$ is odd, one can simply add a zero row at the bottom of $B$ and extend $A$ with a zero row at the bottom and zero column at the right, this does not affect the validity of \eqref{RealFormPRO} and only the controllability in \eqref{RealFormPROcon} falls away. Since $A=-A^T$ is a real matrix, all nonzero eigenvalues are on $i\BR$ and come in complex conjugate pairs, while $\dim\kr A$ is also even, since $n$ is even. Let $i\om_1,\ldots,i\om_s$ be the eigenvalues on $i\BR_+$. If $\om_j\neq0$, then let $k_j$ be the pole-multiplicity of $i\om_j$, while $k_j =(\dim\kr A)/2$ if $\om_j=0$. Then $2\sum_{j=1}^s k_j=n$. Also, there exists an orthogonal matrix $U$ so that
\[
U^T A U =\diag (A_1,\ldots, A_s),\quad \mbox{with } A_j= \om_j \mat{cc}{0&  I_{k_j}\\ - I_{k_j} & 0}.
\]
Now decompose $U^T B$ accordingly as
\[
U^T B=\mat{c}{B_1\\ \vdots\\ B_s}, \quad \mbox{with } B_j\in \BR^{2 k_j \times m}.
\]
Set $Q=M$, $R=D$, $Q_j=B_j^T B_j$ and $R_j = B_j^T A_j B_j$ for $j=1,\ldots,s$. We then have
\begin{align*}
F(z) &=z M +D + B^T (z I_n -A)^{-1}B = zQ + R + B^T U (z I_n -U^TAU)^{-1}U^TB\\
& = zQ + R + \sum_{j=1}^s B_j^T (zI_{2k_j}-A_j)^{-1}B_j\\
& = zQ + R + \sum_{j=1}^s B_j^T \mat{cc}{zI_{k_j}&-\om_j I_{k_j}\\ \om_j I_{k_j} & z I_{k_j}}^{-1}B_j \\
& = zQ + R + \sum_{j=1}^s \frac{1}{z^2 + \om_j^2} B_j^T \mat{cc}{zI_{k_j}&\om_j I_{k_j}\\ -\om_j I_{k_j} & z I_{k_j}}B_j\\
&= zQ + R + \sum_{j=1}^s \frac{1}{z^2 + \om_j^2} B_j^T (z I_{2k_j}+ A_j)B_j\\
&= zQ + R + \sum_{j=1}^s \frac{1}{z^2 + \om_j^2} (zQ_j + R_j).
\end{align*}
Hence \eqref{FosterMat2} holds. For $j=1,\ldots,s$, $Q,Q_j\geq 0$ and $R,R_j$ are skew-symmetric, since $A_j$ is skew-symmetric for each $j$. Furthermore, we have
$ -\om_j I_{2k_j} \leq i A_j \leq \om_j I_{2k_j}$, which provides \eqref{FosterCon2} after multiplying by $B_j$ on the right and $B_j^T$ on the left.
\end{proof}

For our proof of the sufficiency of \eqref{FosterMat2} and \eqref{FosterCon2} we require the following lemma. We note here that in \cite{tHN20} a shorter, though less constructive, proof of the sufficiency of \eqref{FosterMat2} and \eqref{FosterCon2} was given, using the convex invertible cone structure of $\PRO_m$. The advantage of the proof given here is that it enables us to explicitly construct a realization as in Theorem \ref{T:Transfer} starting from the Foster representation formula \eqref{FosterMat2}.

\begin{lemma}\label{L:Lift}
Let $\om>0$, $\BR^{m\times m}\ni Q\geq 0$ and $R\in\BR^{m\times m}$ skew-symmetric so that
\begin{equation}\label{QRineq}
-\om Q\leq iR \leq \om Q.
\end{equation}
Then there exists an integer $q\geq 0$ and $B\in\BR^{2q\times m}$ so that
\begin{equation}\label{QRfact}
Q=B^TB\quad \mbox{and}\quad R=B^TAB\quad \mbox{with}\quad  A=\sbm{0& \om I_q\\ -\om I_q & 0}
\end{equation}
and $(A,B)$ is a controllable pair.
\end{lemma}

\begin{proof}[\bf Proof]
Throughout the proof, for any matrix $C$ we define $\La_C=\sbm{0& C\\ -C^T& 0}$. Assume $\om$, $Q$ and $R$ are as in the lemma. Factor $Q=B_0^TB_0$ with $B_0\in\BR^{p \times m}$ and $p=\rank Q$. Then $B_0$ is right-invertible. We write $B_0^{+}$ for the Moore-Penrose right inverse of $B_0$. Set $S_0=(B_0^+)^T R B_0^+\in\BR^{p \times p}$. Note that \eqref{QRineq} implies that $\kr Q\subset \kr R$ and $\ran R \subset \ran Q$. Therefore, we have
\[
B_0^T S_0 B_0= B_0^T (B_0^+)^T R B_0^+ B_0= P_{\im Q} R P_{\kr Q^\perp}=R.
\]
Moreover, $S_0$ is skew-symmetric and \eqref{QRineq} implies $-\om I_p\leq i S_0 \leq \om I_p$. In particular, the eigenvalues of $S_0$ come in complex conjugate pairs $(i\al,-i\al)$ with $\al\in[0,\om]$, except possibly 0 which may have odd multiplicity.

We first consider the case that $p$ is even, say $p=2 k$. Then there exists an orthogonal matrix $U_0\in\BR^{p\times p}$ so that $S_0=U_0\diag(\La_{\al_1},\ldots,\La_{\al_k})U_0^T$ with $\om\geq\al_1\geq \ldots\geq \al_k\geq 0$ so that $i\al_j$, $j=1,\ldots,k$ are the eigenvalues of $S_0$ on $i\BR_+$, multiplicities taken into account. Define $B_1=U_0^T B_0$ and $A_1=\diag(\La_{\al_1},\ldots,\La_{\al_k})$. Then $B_1^T B_1=Q$ and $B_1^T A_1 B_1=B_0^T S_0 B_0=R$.

Let $1\leq l\leq k$ be so that $\al_1=\cdots = \al_l=\om$ and $\al_{l+1}<\om$, setting $l=0$ in case $\al_1<\om$ and $l=k$ if $\al_j=\om$ for all $j$. Set $q=l+2(k-l)=k+(k-l)$ and define $A=\La_{\om I_q}$ in $\BR^{2q \times 2q}$ as above. Then there exists a permutation matrix $W$ so that $W^T AW=\diag(\La_{\al_1},\ldots \La_{\al_l}, \La_{\om I_2},\ldots,\La_{\om I_2})=:A_2$, using $\al_1=\cdots = \al_l=\om$.

For $j=l+1,\ldots,k$  set $U_j=\om^{-1}\sbm{\al_j & \eta_j\\ -\eta_j & \al_j}$, where $\eta_j=(\om^2-\al_j^2)^{1/2}$, except if $\om=0$ when we set $U_j=\La_{1}$, and define $\what{U}_j=\sbm{U_j & 0\\ 0& I_2}$. Note that both $U_j$ and $\what{U}_j$ are orthogonal and we have $\La_{\om U_j}= \what{U}_j \La_{\om I_2} \what{U}_j^T$. Now define the orthogonal matrix $\what{U}=\diag(I_2,\ldots,I_2, \what{U}_{l+1}^T,\ldots,\what{U}_k^T)$ in $\BR^{2q\times 2q}$. Then
\[
\what{U}^T W^T A W \what{U}=\what{U}^T A_2 \what{U}=\diag(\La_{\al_1},\ldots \La_{\al_l},\La_{\om U_{l+1}},\ldots,\La_{\om U_{k}})=:\wtilA_1\in\BR^{2q\times 2q}.
\]

Note that $A_1$ can be obtained by compressing $\wtilA_1$ to the rows and columns indexed by $1,\ldots,2l,2l+1,2l+3,\ldots,2l+4(k-l)-1=2q-1$. Let $b_j$, $j=1,\ldots,2k$, be the $j$-th row of $B_1$. We now extend $B_1$ to a matrix  $\wtilB_1\in\BR^{2q\times m}$ by
\begin{equation}\label{tilB1}
\wtilB_1^T=\mat{cccccccccc}{b_1^T &\cdots&b_{2l}^T&b_{2l+1}^T&0&b_{2l+2}^T&0&\cdots&b_{2k}^T&0}.
\end{equation}
Then we have $Q=\wtilB_1^T\wtilB_1$ and $R=\wtilB_1^T \wtilA_1 \wtilB_1$. Now set $B=W\what{U}\wtilB_1$. Since $A=W\what{U}\wtilA_1\what{U}^T W^T$, with $W$ and $\what{U}$ orthogonal, we find that \eqref{QRfact} holds. Hence, it remains to show that the pair $(A,B)$  is controllable. Note that $A^jB= W\what{U}\wtilA_1^j \wtil{B}_1$. Therefore, it is equivalent to show $(\wtilA_1,\wtil{B}_1)$ is a controllable pair. Note that
\begin{equation}\label{mincomp1}
\mat{c}{\wtilB_1^T\\\wtilB_1^T\wtilA_1^T}=\mat{cccccc}{L_1&\cdots& L_l&H_{l+1}&\cdots&H_{q}}
\end{equation}
with for $j=1,\ldots,l$ and $s=l+1,\ldots,q$ we define
\[
 L_j=\mat{cc}{b_{2j-1}^T & b_{2j}^T\\ \om b_{2j}^T & \om b_{2j-1}^T},\, H_s=\mat{cc}{b_{2s-1}^T &0\\\al_s b_{2s}^T & -\eta_{s} b_{2s}^T}\in\BR^{2m \times 2}.
\]
By construction $\{b_1^T,\ldots,b_{2k}^T\}$ forms a set of linearly independent vectors. Hence if the matrix \eqref{mincomp1} were to have linearly dependent columns, they must be among the columns indexed by $2l+2,2l+4,\ldots,2q$. However, this can also not occur, since $\eta_{l+1},\ldots,\eta_{2k}\neq 0$ and $\{b_{2l+1}^T,\ldots,b_{2k}^T\}$ is a set of linearly independent vectors. This shows that $[\wtilB_1\  \wtilA_1\wtilB_1]$ has full row rank, hence $[\wtilB_1\  \wtilA_1\wtilB_1\  \cdots\  \wtilA_1^{2q-1}\wtilB_1]$ has full row rank, provided $q\geq 1$. In case $q=0$, controllability is trivial. Hence we find that the pair $(A,B)$ is controllable.

Finally, we consider the case where $p=\rank(Q)$ is odd, say $2k+1$. The above procedure can be followed with a few modifications. We have $S_0=U_0 A_1 U_0^T$ where now $A_1=\diag(\La_{\al_1},\ldots,\La_{\al_k},0)\in\BR^{2k+1 \times 2k+1}$. Set
\begin{align*}
A_2&=\diag(\La_{\al_1},\ldots,\La_{\al_l},\La_{\om I_2},\ldots,\La_{\om I_2},\La_\om),\\
\what{U}&=\diag(I_2,\ldots,I_2, \what{U}_{l+1}^T,\ldots,\what{U}_k^T, I_2),
\end{align*}
hence we add $2\times 2$ diagonal blocks $\La_\om$ and $I_2$, respectively. Next define $\wtilA_1=\what{U}^T A_2\what{U}$ and extend $B_1=U_0^T B_0$ to $\wtilB_1$ as in \eqref{tilB1} except that $\wtilB_1$ now has $b_{2k+1}$ and 0 as its last two rows. It is easy to see that $A_2=W^T A W$ holds for some permutation matrix $W$ and $A$ as in \eqref{QRfact} where now $q=l+2(k-l)+1$. Following the remainder of the proof for the case where $\rank(Q)$ is even, with $B=W\what{U}\wtilB_1$, we see that \eqref{QRfact} holds and that $(A,B)$ is controllable, because $(\wtilA_1,\wtilB_1)$ is controllable. For the latter, note that $b_1,\ldots,b_{2k+1}$ are linearly independent vectors and in the above matrix $\sbm{\wtilB_1^T\\ \wtilB_1^T \wtilA_1^T}$ after the modification of the present paragraph the two columns $\sbm{b_{2k+1}^T & 0\\ 0 & \om b_{2k+1}^T }$ are added leading to a new $\sbm{\wtilB_1^T\\ \wtilB_1^T \wtilA_1^T}$ that still has full column rank.
\end{proof}

\begin{proof}[\bf Proof of sufficiency part of Theorem \ref{T:Main2}]
Using the previous lemma, we now show how a realization as in Theorem \ref{T:Transfer} of a $F\in\PRO_m$ can be obtained from its Foster representation. Hence, assume $F\in\PRO_m$ is given by \eqref{FosterMat2}--\eqref{FosterCon2}. Without loss of generality $\om_j\neq \om_k$ if $j\neq k$. For $j=1,\ldots,s$ apply the factorization from Lemma \ref{L:Lift}, i.e., $Q_j=B_j^TB_j$ and $R_j=B_j^T A_j B_j$ with $A_j=\sbm{0& \om_j I_{q_j}\\ -\om_j I_{q_j} &0}$. We get
\begin{align*}
\frac{1}{z^2 + \om_j^2} \left(zQ_j + R_j \right)
& = \frac{1}{z^2 + \om_j^2}B_j^T \left(zI_{2q_j} + A_j \right) B_j
= \frac{1}{z^2 + \om_j^2}B_j^T  \sbm{zI_{q_j} & \om_j I_{q_j}\\ -\om_j I_{q_j} & zI_{q_j} }  B_j\\
&= B_j^T  \sbm{zI_{q_j} & -\om_j I_{q_j}\\ \om_j I_{q_j} & zI_{q_j} }^{-1}  B_j
= B_j^T \left(zI_{2q_j} - A_j \right)^{-1} B_j.
\end{align*}
Now set $M=Q$, $D=R$, $A=\diag(A_1,\ldots,A_s)$ and $B^T=[B_1^T\, \cdots\, B_s^T]$. It is clear from the above computation that $F$ in \eqref{FosterMat2} is also given by \eqref{RealFormPRO} with this choice of $M$, $D$, $A$ and $B$. To see that the pair $(A,B)$ is controllable, note that
\[
\mat{cc}{A-\la I & B}=\mat{ccccc}{A_1-\la I& 0 & \cdots & 0 & B_1\\
0&A_2-\la I&&&B_2\\
\vdots&\ddots&\ddots&\vdots&\vdots\\
\vdots&\ddots&A_{s-1}-\la I&0&B_{s-1}\\
0&\cdots&0&A_s-\la I&B_s}.
\]
Clearly, for $\la\neq \pm i\om_j$, for $j=1,\ldots,s$, the matrix has full row rank. For $\la=i\om_j$ or $\la=-i\om_j$ the $(k,k)$ block entries for $k\neq j$ are still invertible, since $\om_j\neq \om_k$, and the rows in the $j$-th block row are independent because $(A_j,B_j)$ is a controllable pair. Hence $(A,B)$ is a controllable pair, as claimed.
\end{proof}

\begin{remark}\label{R:pole-multi}
Apart from a concrete procedure to determine a minimal realization for $F$ explicitly from the Foster representation, the above proof also shows how the pole-multiplicities can be computed. For the pole at $\infty$ it is clear its multiplicity is $\rank Q$. Fix a finite pole $i\om_j$ and let $R_j$ and $Q_j$ be as in \eqref{FosterMat2}. In this case, the multiplicity of $\om_j$ in not necessarily equal to $\rank Q_j$, but rather the size of the matrix $A_j$ obtained from the construction of Lemma \ref{L:Lift}. Set $p_j=\rank Q_j$ and determine a factorization $Q_j=B_{0,j}^TB_{0,j}$ with $B_{j,0}\in\BR^{p_j\times m}$, which is unique up to multiplication with a $p_j\times p_j$ unitary matrix. Set $S_{j}:= (B_{0,j}^+)^T R_j B_{0,j}^+$, with $B_{0,j}^+$ the Moore-Penrose right-inverse of $B_{0,j}$. Then $S_j$ is skew-symmetric and all eigenvalues of $S_j$ on $i\BR_+$ are bounded by $i\om_j$. Let $l_j$ be the number of eigenvalues equal to $i\om_j$. Then the pole-multiplicity of $\om_j$ is equal to $l_j+ 2(p_j/2 -l_j)$ in case $p_j$ is even and $l_j+ 2((p_j+1)/2 -l_j)$ in case $p_j$ is odd.
\end{remark}

\section{Inversion}\label{S:Inverse}

Since $\PRO_m$ is a convex invertible cone, for a function $F\in\PRO_m$, it follows that $F^{-1}$ is also in $\PRO_m$, provided $F$ is invertible, i.e., $\det F(z)\not \equiv 0$. In this section, we determine when $F\in \PRO_m$ is invertible and provide realization formulas for its inverse, of the form as in Section \ref{S:Transfer}, in case $F$ is invertible. Throughout this section we shall assume $F$ is given in the transfer function form of Theorem \ref{T:Transfer}, that is,
\begin{equation}\label{RealFormPRO2}
F(z)=z M +D + B^T (z I_n -A)^{-1}B,
\end{equation}
for some integer $n\geq 0$, $M,D\in\BR^{m \times m}$, $B\in\BR^{n \times m}$ and $A\in\BR^{n \times n}$ with
\begin{equation}\label{RealFormPROcon2}
M\geq 0,\quad A^T=-A,\quad D^T=-D,\quad \mbox{$(A,B)$ controllable.}
\end{equation}

By the inversion result for descriptor systems from \cite{MPR07}, we obtain the following characterization for invertibility of $F$ and of its inverse.

\begin{proposition}\label{P:inv1}
Let $F\in\PRO_m$ be given by \eqref{RealFormPRO2}-\eqref{RealFormPROcon2}. Then for any $z\in\BC$ we have
\begin{equation}\label{Regular1}
\det F(z)\not\equiv 0
 \ \ \Longleftrightarrow\ \
\det \left( \mat{cc}{zI_n & 0 \\ 0 & zM} -\mat{cc}{A & B \\ -B^T &-D} \right)\not\equiv 0.
\end{equation}
Moreover, in that case we have
\begin{equation}\label{Inverse1}
F(z)^{-1}=\mat{cc}{0 & I_m} \left( \mat{cc}{zI_n & 0 \\ 0 & zM} -\mat{cc}{A & B \\ -B^T &-D} \right)^{-1}\mat{c}{0\\ I_m}.
\end{equation}
\end{proposition}

\begin{proof}[\bf Proof]
From \eqref{RealFormPRO2}-\eqref{RealFormPROcon2} one obtains the descriptor realization form \eqref{Weierstrass1}-\eqref{Weierstrass2}, where $q=\rank M$ and $K\in\BR^{q\times m}$ is so that $K^TK=M$. By the inversion formula for descriptor systems from Theorem 3.1 in \cite{MPR07} it follows that
\begin{align*}
&F(z)^{-1}
=\mat{cc}{0 & I_m} \left(z\mat{cc}{E^\circ &0\\0&0} - \mat{cccc}{A^\circ & B^\circ\\ C^{\circ T} & D^\circ} \right)^{-1} \mat{c}{0\\ -I_m}\\
 &= \mat{cccc}{0 & 0 & 0 & I_m}
  \mat{cccc}{zI_n-A&0&0& -B\\ 0&-I_q&zI_q&0\\ 0&0&-I_q&K\\ -B^T&-K^T&0&-D}^{-1} \mat{c}{0\\ 0\\ 0\\ -I_m}\\
  &= \mat{cc|cc}{0 & I_m & 0 & 0}
  \mat{cc|cc}{zI_n-A&-B&0& 0\\ -B^T& -D&-K^T&0\\ \hline  0& 0&-I_q&zI_q\\ 0&K&0&-I_q}^{-1}
\mat{c}{0\\ -I_m\\ \hline 0\\ 0}
\end{align*}
and $\det F(z)\not\equiv 0$ precisely when the $4 \times 4$ block matrix is invertible. Since the right lower $2 \times 2$  block $\sbm{-I&z I\\ 0&-I}$ is invertible for all $z$, it follows that the above inverse exists if and only if the Schur complement with respect to this $2 \times 2$ block:
\begin{align*}
\De(z)&:=\mat{cc}{zI_n-A & -B\\ -B^T&-D}-\mat{cc}{0&0\\-K^T & 0}\mat{cc}{-I_q&z I_q\\ 0&-I_q}^{-1}\mat{cc}{0& 0\\ 0&  K}\\
&=z\mat{cc}{I_n&0\\0&-M}- \mat{cc}{A&B\\B^T &D}
\end{align*}
is invertible. Via the standard Schur complement inversion formula, cf., \cite{Z05}, one now obtains that
\begin{align*}
F(z)^{-1}
&= \mat{cc|cc}{0 & I_m & 0 & 0}
  \mat{c|c}{\De(z)^{-1} & *\\ \hline * & *}
\mat{c}{0\\ -I_m\\ \hline 0\\ 0}\\
&=\mat{cc}{0&I_m} \left(z\mat{cc}{I_n&0\\0&-M}- \mat{cc}{A&B\\B^T &D}\right)^{-1} \mat{c}{0 \\ -I_m}\\
&=\mat{cc}{0&I_m} \left(z\mat{cc}{I_n&0\\0&M}- \mat{cc}{A&B\\-B^T &-D}\right)^{-1} \mat{c}{0 \\ I_m},
\end{align*}
which proves our claim.
\end{proof}

Next we provide an easily verifiable criteria to determine when $\det F(z)\not\equiv 0$.

\begin{lemma}\label{L:InvCon}
Let $F\in\PRO_m$ be given by \eqref{RealFormPRO2}-\eqref{RealFormPROcon2}. Then $\det F(z)\not\equiv 0$ if and only if $\kr\left(\sbm{B\\D}|_{\kr M}\right)=\{0\}$.
\end{lemma}

\begin{proof}[\bf Proof]
In Proposition \ref{P:inv1} we noted that $\det F(z)\not\equiv 0$ precisely when the pair $(\whatE,\whatA)$ with $\whatE=\sbm{I_n&0\\0&M}$ and $\whatA=\sbm{A&B\\-B^T &-D}$ is regular, that is, $\det(z\what{E} - \what{A})\not\equiv 0$. The claim now follows immediately from the following lemma.
\end{proof}

\begin{lemma}\label{L:RegLem}
Let $\whatE\geq 0$ and $\whatA\in\BR^{k\times k}$ skew-symmetric. Then the pair $(\whatE,\whatA)$ is regular if and only if $\kr (\whatA|_{\kr \whatE})=\{0\}$.
\end{lemma}

\begin{proof}[\bf Proof]
For the necessity, just note that $\kr (\whatA|_{\kr \whatE})$ is contained in $\kr (z\whatE-\whatA)$ for all $z\in\BC$. So it remains to prove sufficiency. Assume $\kr (\whatA|_{\kr \whatE})=\{0\}$. Decompose $\BR^k=\im\, \whatE \oplus \kr \whatE$. Further decompose $\kr \whatE=\cZ_3\oplus\cZ_4$ with $\cZ_3=\im(P_{\kr \whatE} \whatA |_{\kr \whatE})$ and $\cZ_4=\kr(P_{\kr \whatE} \whatA |_{\kr \whatE})$ and $\im\, \whatE=\cZ_1\oplus\cZ_2$ with $\cZ_1=\whatA (\cZ_4)$ and $\cZ_2=\im\, \whatE \ominus \cZ_1$. Note that $\whatA$ maps $\cZ_4$ into $\im\, \whatE$ by definition of $\cZ_4$. Using that $\whatE$ is positive semidefinite and $\whatA$ skew-symmetric, we now obtain that with respect to the decomposition $\BR^k=\cZ_1\oplus\cZ_2\oplus\cZ_3\oplus\cZ_4$, the matrices $\whatE$ and $\whatA$ have the following form
\[
\whatA=\mat{cccc}{A_{11}&A_{12}&A_{13}&A_{14}\\ -A_{12}^T&A_{22}&A_{23}&0\\ -A_{13}^T&-A_{23}^T&A_{33}&0\\ -A_{14}^T &0&0&0},\quad
\whatE=\mat{cccc}{E_{11}&E_{12}&0&0\\E_{12}^T&E_{22}&0&0\\0&0&0&0\\0&0&0&0},
\]
with $\sbm{E_{11}&E_{12}\\E_{12}^T&E_{22}}$ positive definite, and hence $E_{11}$ and $E_{22}$ positive definite,  $A_{33}$ invertible and $A_{14}$ full row-rank. The assumption $\kr (\whatA|_{\kr \whatE})=\{0\}$ is equivalent to $\kr A_{14}=\{0\}$, hence to $A_{14}$ invertible. Now note that
\[
z\what{E}-\what{A}=
\mat{cccc}{zE_{11}-A_{11}&zE_{12}-A_{12}&-A_{13}&-A_{14}\\ zE_{12}^T+ A_{12}^T&zE_{22}-A_{22}&-A_{23}&0\\ A_{13}^T&A_{23}^T&-A_{33}&0\\ A_{14}^T &0&0&0}.
\]
Since $A_{14}$ is invertible, and hence $A_{14}^T$ is invertible, we obtain that $z\whatE-\whatA$ is invertible if and only if
\[
\mat{cc}{zE_{22}-A_{22}&-A_{23}\\A_{23}^T&-A_{33}}
\]
is invertible. Taking the Schur complement with respect to $-A_{33}$ we see that invertibility of this $2 \times 2$ block matrix is equivalent to invertibility of the Schur complement
\[
zE_{22}-A_{22}- (-A_{23})(-A_{33})^{-1}A_{23}^T=zE_{22}-(A_{22}+A_{23}A_{33}^{-1}A_{23}^T).
\]
Note that $A_{22}$ and $A_{33}$ are skew-symmetric and $E_{22}$ is positive definite. Therefore, $A_{22}+A_{23}A_{33}^{-1}A_{23}^T$ is skew-symmetric, and for any $0\neq z\in\BR$ we have $\det(zE_{22}-(A_{22}+A_{23}A_{33}^{-1}A_{23}^T))\neq 0$. Since either $\det(zE_{22}-(A_{22}+A_{23}A_{33}^{-1}A_{23}^T))\equiv 0$ or there are only finitely many roots, we see that $\det(zE_{22}-(A_{22}+A_{23}A_{33}^{-1}A_{23}^T))\not\equiv 0$. Consequently, we have $\det(z\whatE-\whatA)\not\equiv 0$, hence the pair $(\whatE,\whatA)$ is regular.
\end{proof}

The realization \eqref{Inverse1} will in general not be minimal, and hence some of the poles of the resolvent may not be poles of $F^{-1}$, or the multiplicities may be inflated. To obtain a minimal realization, we decompose the matrices $M$, $D$ and $B$ with respect to the decomposition of $\BR^m$ given by
\begin{equation}\label{Rmdec}
\begin{aligned}
&\hspace*{2cm}\BR^m=\cX_1\oplus\cX_2\oplus\cX_3,\mbox{ with }\\
&\cX_1=\kr M^\perp,\ \cX_2= \kr (P_{\kr M} D|_{\kr M})^\perp,\ \cX_3=\kr (P_{\kr M} D|_{\kr M}),
\end{aligned}
\end{equation}
which yields decompositions of the form
\begin{equation}\label{BDMdec}
B\!=\!\mat{ccc}{B_1&B_2&B_3},\
D\!=\!\mat{ccc}{D_{11}&D_{12}&D_{13}\\ -D_{12}^T &D_{22}&0\\ -D_{13}^T&0&0},\
M\!=\!\mat{ccc}{M_1&0&0\\ 0&0&0\\ 0&0&0},
\end{equation}
with $M_1$ and $D_{22}$ invertible. In particular, $M_1$ is positive definite and $D_{22}$ is invertible and real, skew-symmetric, so that $\cX_2$ must have even dimension. We set
\[
m_1=\dim \cX_1,\ \
m_2= \dim\cX_2,\ \
m_3= \dim \cX_3,\ \ \mbox{so that}\ \ m=m_1+m_2+m_3, \mbox{ $m_2$ even}.
\]

As an intermediate step towards our main result, we present a minimal descriptor realization for $F^{-1}$ which is not in Weierstrass form yet. For this purpose, consider linear maps $K_1$ and $\Xi$ so that
\begin{equation}\label{K1Xi}
K_1:\cX_1\to\BR^{m_1},\quad K_1^TK_1=M_1,\qquad \Xi:\cX_3\to\BR^{m_3},\quad \Xi^T \Xi=I_{\cX_3}.
\end{equation}
Note that $K_1$ is invertible and $\Xi$ orthogonal. Further, define
\begin{equation}\label{tilAtilB}
\begin{aligned}
\wtilA&=\mat{cc}{A-B_2 D_{22}^{-1}B_2^T & (B_1+B_2D_{22}^{-1}D_{12}^T)K_1^{-1}\\
K_1^{-T}(-B_1^T+ D_{12}D_{22}^{-1}B_2^T) & - K_1^{-T}(D_{11}+ D_{12}D_{22}^{-1}D_{12}^T)K_1^{-1}},\\
\wtilB&=\mat{c}{B_3\Xi^T\\ -K_1^{-T} D_{13}\Xi^T}.
\end{aligned}
\end{equation}
In terms of the decomposition \eqref{Rmdec}--\eqref{BDMdec}, the condition for $\det F(z)\not\equiv 0$ of Lemma \ref{L:InvCon} translates to $\kr \sbm{B_3\\ D_{13}}=\{0\}$, or, equivalently, $\kr \wtilB=\{0\}$.

\begin{proposition}\label{P:inv2}
Let $F\in\PRO_m$ be given by \eqref{RealFormPRO2}--\eqref{RealFormPROcon2} and decompose $B$, $D$, $M$ as in \eqref{BDMdec}. Define $\wtilA$ and $\wtilB$ as in \eqref{tilAtilB} with $K_1$ and $\Xi$ as in \eqref{K1Xi}. Assume $\kr \wtilB=\{0\}$, so that $\det F(z)\not\equiv 0$.  Then
\begin{equation}\label{FinvForm}
F(z)^{-1}= \what{D}_\tu{inv} + \what{B}^{T}_\tu{inv} (z \what{E}_\tu{inv} - \what{A}_\tu{inv})^{-1} \what{B}_\tu{inv},
\end{equation}
where
\begin{equation}\label{whatABCD}
\begin{aligned}
 &\qquad\what{E}_\tu{inv}  = \mat{cc}{I_{n+ m_1}&0\\ 0&0},\quad  \what{A}_\tu{inv}  = \mat{cc}{\wtilA & \wtilB\\ -\wtilB^T &0}, \\
 & \what{B}_\tu{inv}  = \mat{ccc}{0 & B_2D_{22}^{-1} & 0 \\ K_1^{-T} &  -K_1^{-T} D_{12} D_{22}^{-1} & 0 \\ 0 & 0 & \Xi}, \quad
  \what{D}_\tu{inv}  = \mat{ccc}{ 0 & 0 & 0 \\ 0 & D_{22}^{-1} & 0 \\ 0 & 0 & 0 },
\end{aligned}
\end{equation}
and the descriptor realization \eqref{FinvForm} of $F^{-1}$ is minimal.
\end{proposition}

\begin{proof}[\bf Proof] Set
\begin{align*}
T_1&=\mat{cccc}{I_n&0&B_2 D_{22}^{-1}&0\\0&K_1^{-T}&-K_1^{-T}D_{12}D_{22}^{-1}&0\\0&0&0&\Xi\\0&0&I_{\cX_2}&0},
\end{align*}
and note that $T_1$ is invertible.
A straightforward computation shows that
\[
T_1\mat{cc}{zI_n-A & -B\\ B^T& zM+ D} T_1^T=\mat{cc}{z \whatE_\tu{inv}-\whatA_\tu{inv} & 0\\ 0 & D_{22}}.
\]
Since $D_{22}$ is invertible, it follows that $z \whatE_\tu{inv}-\whatA_\tu{inv}$ is invertible if and only if $\sbm{zI-A & -B\\ B^T& zM+ D}$ is invertible. Applying this transformation to the formula for $F^{-1}$ in \eqref{Inverse1} we obtain that
\begin{align*}
F(z)^{-1}
& =\mat{cc}{0 & I_m}\mat{cc}{zI_n-A & -B\\ B^T& zM+ D}^{-1}\mat{c}{0\\ I_m}\\
& =\mat{cc}{0 & I_m}T_1^T \mat{cc}{z \whatE_\tu{inv}-\whatA_\tu{inv} & 0\\ 0 & D_{22}}^{-1} T_1\mat{c}{0\\ I_m}\\
&= \mat{ccc|c}{0&K_1^{-1}&0&0\\-D_{22}^{-1}B_2^T&D_{22}^{-1}D_{12}^TK_1^{-1}&0&I_{\cX_2}\\0&0&\Xi^T&0}\times\\
&\qquad\qquad \times \mat{cc}{(z \whatE_\tu{inv}-\whatA_\tu{inv})^{-1} & 0\\ 0 & D_{22}^{-1}}
\mat{ccc}{0&B_2 D_{22}^{-1}&0\\K_1^{-T}&-K_1^{-T}D_{12}D_{22}^{-1}&0\\0&0&\Xi\\\hline 0&I_{\cX_2}&0}\\
&= \mat{ccc}{0&0&0\\0&D_{22}^{-1}&0\\0&0&0}+
\mat{cccc}{0&K_1^{-1}&0\\-D_{22}^{-1}B_2^T&D_{22}^{-1}D_{12}^TK_1^{-1}&0\\0&0&\Xi^T}\times\\
&\qquad\qquad\qquad \times(z\whatE_\tu{inv}-\whatA_\tu{inv})^{-1}
\mat{ccc}{0&B_2 D_{22}^{-1}&0\\K_1^{-T}&-K_1^{-T}D_{12}D_{22}^{-1}&0\\0&0&\Xi}\\
&=\what{D}_\tu{inv} + \what{B}^{T}_\tu{inv} (z \what{E}_\tu{inv} - \what{A}_\tu{inv})^{-1} \what{B}_\tu{inv}.
\end{align*}
Hence, we established \eqref{FinvForm}. It remains to prove that this descriptor realization is minimal. By Theorem 6.2 from \cite{FJ04}, see also Section \ref{S:TransferPre}, the descriptor realization \eqref{FinvForm} is minimal if and only if the following five conditions are met:
\begin{itemize}

\item[(i)] $\rank \mat{cc}{z \whatE_\tu{inv} - \whatA_\tu{inv} & \whatB_\tu{inv}}=n+m_1+m_3$ for all $z\in\BC$;

\item[(ii)] $\rank \mat{cc}{\whatE_\tu{inv} & \whatB_\tu{inv}}=n+m_1+m_3$;

\item[(iii)] $\rank \mat{cc}{z \whatE_\tu{inv}^T - \whatA_\tu{inv}^T & \whatB_\tu{inv}}=n+m_1+m_3$ for all $z\in\BC$;

\item[(iv)] $\rank \mat{cc}{\whatE_\tu{inv}^T & \whatB_\tu{inv}}=n+m_1+m_3$;

\item[(v)]$\whatA_\tu{inv} (\kr \whatE_\tu{inv})\subset \im\, \whatE_\tu{inv}$.
\end{itemize}
Since $\whatA_\tu{inv}^T =- \whatA_\tu{inv}$ and $\whatE_\tu{inv}^T=\whatE_\tu{inv}$, conditions (iii) and (iv) follow from (i) and (ii), hence it suffices to verify (i), (ii) and (v). From the formulas of $\whatE_\tu{inv}$ and $\whatB_\tu{inv}$ it is clear that $\rank \sbm{\whatE_\tu{inv} & \whatB_\tu{inv}}=n+m_1+\rank \Xi=n+m_1+m_3$, since $\Xi$ is a orthogonal map, hence (ii) holds. Also, $\whatA_\tu{inv} (\kr \whatE_\tu{inv})=\im \sbm{\wtilB\\ 0}\subset \BR^{n+m_1}\oplus\{0\}=\im\, \whatE_\tu{inv}$. Thus (v) is also satisfied, and it remains to prove (i).  First note that
\begin{align*}
  \rank \mat{ccc}{zI_n-A &-B&0\\ B^T&zM+D&I_m} &= \rank \mat{cc}{z I_n-A & -B}+m=n+m,
\end{align*}
since $(A,B)$ is assumed to be a controllable pair. Using the invertible matrix $T_1$ defined above we note that
\[
T_1 \mat{ccc}{zI_n-A &-B&0\\ B^T&zM+D&I_m} \mat{cc}{T_1^T&0\\0& I_m}
= \mat{ccc}{z\whatE_\tu{inv}-\whatA_\tu{inv}&0&\whatB_\tu{inv}\\0&D_{22}&R}
\]
with $R=\mat{ccc}{0&I_{\cX_2}&0}$. Consequently, since $D_{22}$ is invertible, we have
\begin{align*}
n+m_1+m_2+m_3&=n+m=\rank\mat{ccc}{z\whatE_\tu{inv}-\whatA_\tu{inv}&0&\whatB_\tu{inv}\\0&D_{22}&R}\\
&=m_2+ \rank\mat{cc}{z\whatE_\tu{inv}-\whatA_\tu{inv}&\whatB_\tu{inv}}.
\end{align*}
Hence,  $\rank \sbm{z\whatE_\tu{inv}-\whatA_\tu{inv}&\whatB_\tu{inv}}=m+m_1+m_3$, as desired.
\end{proof}

We are now ready to present the minimal Weierstrass realization for $F^{-1}$.

\begin{theorem}\label{T:inv1}
Let $F\in\PRO_m$ be given by \eqref{RealFormPRO2}-\eqref{RealFormPROcon2} and decompose $B$, $D$, $M$ with respect to the decomposition \eqref{Rmdec} of $\BR^m$ as in \eqref{BDMdec}. Define $\wtilA$ and $\wtilB$ as in \eqref{tilAtilB}, with $K_1$ and $\Xi$ as in \eqref{K1Xi}, and assume $\kr \wtilB =\{0\}$ so that $\det F(z)\not\equiv 0$. Set $k=n+m_1-m_3$ and let $\Ga\in\BR^{(n+m_1)\times k}$ be an isometry with $\im\, \Ga\perp \im\, \wtilB$. Then a minimal Weierstrass descriptor realization of the inverse of $F$ is given by
\begin{equation}\label{FinvWeierstrass}
F(z)^{-1}=D_\tu{inv}^\circ +C_\tu{inv}^{\circ T} (zE_\tu{inv}^\circ -A_\tu{inv}^\circ)^{-1} B_\tu{inv}^\circ
\end{equation}
with
\begin{align}
&E_\tu{inv}^\circ\!=\!\mat{ccc}{\! I_k&0&0\!\\\! 0&0& I_{m_3}\! \\\! 0&0&0\!},\   A_\tu{inv}^\circ\!=\!\mat{ccc}{\! A_\tu{inv}&0&0\! \\\! 0&I_{m_3}&0\! \\\! 0&0&I_{m_3}\!},
\  B_\tu{inv}^\circ\!=\!\mat{c}{\! B_\tu{inv}\!\\\! 0\! \\\! -K_\tu{inv}\!}, \label{FinvWeierMats1}\\
&C_\tu{inv}^\circ\!=\!\mat{c}{\!\! B_\tu{inv}\!\! \\\!\! K_\tu{inv}\!\! \\\!\! 0\!\!},\
D_\tu{inv}^\circ \!=\!\mat{ccc}{
\! 0 \!&\! 0 \!&\!  -M_1^{-1}D_{13} \Phi_{33}^{-1}\! \\
\! 0  \!&\!  D_{22}^{-1}  \!&\!  -D_{22}^{-1}\Phi_{23} \Phi_{33}^{-1}\! \\
\! \Phi_{33}^{-1} D_{13}^T M_{1}^{-1}  \!&\!  -\Phi_{33}^{-1}\Phi_{23}^T D_{22}^{-1}   \!&\!
-\Phi_{33}^{-1} \Xi^T \wtilB^T \wtilA \wtilB \Xi \Phi_{33}^{-1}\!},\notag
\end{align}
where we define
\begin{equation}\label{FinvWeierMats2}
\begin{aligned}
& A_\tu{inv}=\Ga^T \wtilA \Ga,\qquad K_\tu{inv}=\mat{ccc}{0&0&-\Xi\Phi_{33}^{-1/2    }},\\
& B_\tu{inv}=\Ga^T
\sbm{0 &B_2 D_{22}^{-1}& (AB_3 -B_1M_1^{-1}D_{13} -B_2 D_{22}^{-1}\Phi_{23})\Phi_{33}^{-1}\\
K_1^{-T} & - K_1^{-T} D_{12}D_{22}^{-1} & - K_1^{-T}(B_1^TB_3 -D_{11}M_1^{-1}D_{13} - D_{12}D_{22}^{-1}\Phi_{23})\Phi_{33}^{-1}},\\
& \Phi_{33}=B_3^T B_3 +D_{13}^T M_1^{-1}D_{13},\quad
\Phi_{23}=B_2^T B_3 +D_{12}^T M_1^{-1}D_{13},
\end{aligned}
\end{equation}
and where
\begin{align*}
\Xi^T \wtilB^T \wtilA \wtilB \Xi
&=B_3^TAB_3 -B_3^T B_1 M_1^{-1}D_{13} +D_{13}^T M_1^{-1}B_1^T B_3+\\
&\qquad\qquad\qquad-D_{13}^TM_1^{-1}D_{11}M_1^{-1}D_{13} -\Phi_{23}D_{22}^{-1}\Phi_{23}.
\end{align*}
\end{theorem}

\begin{proof}[\bf Proof]
Consider the realization of $F(z)^{-1}$ in Proposition \ref{P:inv2}. Define $\Ga$ and $\Xi$, as well as $E_\tu{inv}^\circ, A_\tu{inv}^\circ,\ldots,D_\tu{inv}^\circ$ and $\wtilA$ and $\wtilB$ as in the theorem. Let $\wtilB^+:= (\wtilB^T\wtilB)^{-1}\wtilB^T$ be the Moore-Penrose left inverse of $\wtilB$. Set $\Up:=\wtilB (\wtilB^T\wtilB)^{-1/2}$. By definition of $\Ga$ we have $\wtilB^+ \Ga=0$ and $\Up^T \Ga=0$. Moreover, $\Up$ is an isometry and $\mat{cc}{\Ga & \Up}$ is unitary.  Now define the invertible matrices
\begin{align*}
L_1 &= \mat{ccc}{I_k&0& \Ga^T \wtilA (\wtilB^+)^T \\ 0&I_{m_3}& \half \Up^T \wtilA (\wtilB^+)^T \\0 &0& -(\wtilB^T\wtilB)^{-1/2}} \mat{cc}{\Ga^T &0\\ \Up^T &0\\0 & I_{m_3}}
=\mat{cc}{\Ga^T & \Ga^T \wtilA (\wtilB^+)^T\\\Up^T & \half \Up^T \wtilA (\wtilB^+)^T\\0 & -(\wtilB^T\wtilB)^{-1/2}},\\
L_2 & =\mat{ccc}{\Ga& \Up &0\\0&0& I_{m_3}}
\mat{ccc}{I_k&0&0\\0&0&I_{m_3}\\ - \wtilB^+ \wtilA \Ga& (\wtilB^T\wtilB)^{-1/2}   &-\half \wtilB^+ \wtilA \Up}\\
&=\mat{ccc}{\Ga&0&\Up\\- \wtilB^+ \wtilA \Ga& (\wtilB^T\wtilB)^{-1/2} &-\half \wtilB^+ \wtilA \Up}.
\end{align*}
A direct computation, using $\wtilB^T\Ga=0$ and $\Up^T \wtilB\wtilB^+=\Up^T P_{\im \wtilB}=\Up^T$, shows that the matrices $\whatA_\tu{inv}$ and $\whatE_\tu{inv}$ given by \eqref{whatABCD} satisfy
\[
L_1 \whatA_\tu{inv} L_2  = A^\circ_\tu{inv}
\ands
L_1 \whatE_\tu{inv} L_2 = E^\circ_\tu{inv}.
\]
Hence $(z\what{E}_\tu{inv}-\wtilA_\tu{inv})^{-1}= L_2(z E_\tu{inv}^\circ - A_\tu{inv}^\circ)^{-1}L_1$. Note that $\wtilB^T\wtilB=\Xi \Phi_{33}\Xi^T$, so that $(\wtilB^T\wtilB)^\half=\Xi \Phi_{33}^\half\Xi^T$, since $\Xi$ is unitary. Furthermore, one can compute that
\[
\wtilA \wtilB=\mat{c}{AB_3 -B_1M_1^{-1}D_{13} -B_2 D_{22}^{-1}\Phi_{23}\\- K_1^{-T}(B_1^TB_3 -D_{11}M_1^{-1}D_{13} - D_{12}D_{22}^{-1}\Phi_{23})}\Xi^T.
\]
Using these identities it follows that
\[
L_1 \what{B}_\tu{inv} = \mat{c}{B_\tu{inv}\\ R \\ K_\tu{inv}} \ands
\what{B}_\tu{inv}^T L_2 = \mat{ccc} {B_\tu{inv}^T& -K_\tu{inv}^T & R^T}
\]
where
\[
R=\Xi \Phi_{33}^{-1/2}\mat{ccc}{-D_{13}^T M_1^{-1}&\Phi_{23}^T D_{22}^{-1}& \half \Xi^T \wtilB^T \wtilA \wtilB\Xi \Phi_{33}^{-1}}
\]
and a further computation shows that
\begin{align*}
\Xi^T \wtilB^T \wtilA \wtilB\Xi &=B_3^TAB_3 -B_3^T B_1 M_1^{-1}D_{13} +D_{13}^T M_1^{-1}B_1^T B_3\\
&\qquad\qquad\qquad -D_{13}^TM_1^{-1}D_{11}M_1^{-1}D_{13} -\Phi_{23}^TD_{22}^{-1}\Phi_{23}.
\end{align*}
Therefore, we have
\begin{align*}
F(z)^{-1}&=  \what{D}_\tu{inv} + \what{B}^{T}_\tu{inv} (z \what{E}_\tu{inv} - \what{A}_\tu{inv})^{-1} \what{B}_\tu{inv}\\
&=  \what{D}_\tu{inv} + \what{B}^{T}_\tu{inv}L_2 (z E^\circ_\tu{inv} - A^\circ_\tu{inv})^{-1} L_1\what{B}_\tu{inv}\\
&=\what{D}_\tu{inv}+ \mat{ccc} {B_\tu{inv}^T& -K_\tu{inv}^T & R^T} (z E_\tu{inv}^\circ- A_\tu{inv}^\circ)^{-1} \mat{c}{B_\tu{inv}\\ R \\ K_\tu{inv}}\\
&= \what{D}_\tu{inv}+  C_\tu{inv}^{\circ T}(z E_\tu{inv}^\circ- A_\tu{inv}^\circ)^{-1}B_\tu{inv}^\circ+\\
& \qquad\qquad + \mat{cc}{0 & R^T} \mat{cc}{-I_{m_3} &  z I_{m_3} \\ 0 & -I_{m_3}  }^{-1} \mat{c}{R\\ K_\tu{inv}} +\\
& \qquad\qquad\qquad + \mat{cc}{- K_\tu{inv}^T& -R^T} \mat{cc}{-I_{m_3} &  z I_{m_3} \\ 0 & -I_{m_3}  }^{-1} \mat{c}{R\\ 0} \\
&= \what{D}_\tu{inv} -R^T K_\tu{inv}+K_\tu{inv}^T R +  C_\tu{inv}^{\circ T}(z E_\tu{inv}^\circ- A_\tu{inv}^\circ)^{-1}B_\tu{inv}^\circ.
\end{align*}
So we arrive at \eqref{FinvWeierstrass} by noting that
\begin{align*}
  \what{D}_\tu{inv} +K_\tu{inv}^T R- R^T K_\tu{inv} & = D_\tu{inv}^\circ.
\end{align*}
Minimality of the realization \eqref{FinvWeierstrass} follows directly from the minimality of \eqref{FinvForm}.
\end{proof}

Note that the descriptor realization for $F^{-1}$ of Theorem \ref{T:inv1} precisely has the form of the realization in Theorem \ref{T:Weierstrass}. Reversing the argument in Section \ref{S:Transfer}, we also obtain a realization of the type in Theorem \ref{T:Transfer}.

\begin{theorem}\label{T:inv2}
Let $F\in\PRO_m$ be given by \eqref{RealFormPRO2}-\eqref{RealFormPROcon2} and decompose $B$, $D$, $M$ with respect to the decomposition \eqref{Rmdec} of $\BR^m$ as in \eqref{BDMdec}. Assume $\det F(z)\not\equiv 0$. Then
\begin{equation}\label{RealFormPROinv}
F(z)^{-1}=z M_\tu{inv} +D_\tu{inv} + B_\tu{inv}^T (z I_n -A_\tu{inv})^{-1}B_\tu{inv},
\end{equation}
where $B_\tu{inv}$ and $A_\tu{inv}$ are as in \eqref{FinvWeierMats2}, $D_\tu{inv}=D^\circ_\tu{inv}$ and $M_\tu{inv}=K^T_\tu{inv} K_\tu{inv}$ with $D^\circ_\tu{inv}$ as in \eqref{FinvWeierMats1} and $K_\tu{inv}$ as in \eqref{FinvWeierMats2}. Moreover, the pair $(A_\tu{inv},B_\tu{inv})$ is controllable.
\end{theorem}

\section{Poles and zeros of $\PRO_m$ functions}\label{S:Interlace}

In the scalar case, i.e., $m=1$, the poles and zeros of functions in $\PRO$ interlace on the imaginary axis. This follows easily from the Foster representation \eqref{FosterScalar} for $\PRO$. For $m>1$ the situation is more complicated, yet still a (partial) analogue of the scalar result can be obtained. We shall assume $F\in\PRO_m$ is given by the realization formula of Theorem \ref{T:Transfer} so that $F^{-1}$ admits a realization as in Theorem \ref{T:inv2}. Recall that the zeros of $F$ are defined as the poles of $F^{-1}$, hence, for finite zeros, as the eigenvalues of $A_\tu{inv}$ with multiplicities equal to the dimensions of the corresponding eigen spaces. Hence, for finite poles and zeros one has to analyse the spectrum of $A_\tu{inv}$, in relation to the spectrum of $A$. At $\infty$ the situation is reasonably straightforward, the pole-multiplicity of $F$ is given by $\rank M$ while the zero-multiplicity of $F$ is equal to $\rank M_\tu{inv}=\rank \Phi_{33}=\rank \wtilB=m_3$.
There are three steps from $A$ to $A_\tu{inv}$ that influence the eigenvalues:
\begin{itemize}
\item[(i)] The perturbation from $A$ to $\what{A}:=A-B_2D_{22}^{-1}B_2^T$ via a perturbation of at most rank $m_2$;

\item[(ii)] The extension of $\what{A}$ to $\widetilde{A}=\sbm{\what{A} & \star \\ \star & \star}\in\BR^{(n+m_1)\times (n+ m_1)}$ in \eqref{tilAtilB};

  \item[(iii)] The compression from $\widetilde{A}$ to $A_\tu{inv}\in\BR^{(n+m_1-m_3)\times (n+ m_1-m_3)}$ in \eqref{FinvWeierMats2}.

\end{itemize}
In general, all three steps can occur. However, for $m=1$, step (i) cannot occur, since $m_2$ must be even, but both steps (ii) and (iii) can occur separately, but not the combination of the two, hence there are only two cases to analyse. For $m=2$ the situation is already more complicated, step (i) can occur, but not together with steps (ii) and (iii), however steps (ii) and (iii) can happen separately, but also together, leading to four cases. In \cite{tHN20} we included an analysis of the various cases that occur for $m=1$ and $m=2$.

Here we present a partial analogue of the results in \cite{tHN20} for the general case. This requires some variational principles for eigenvalues of Hermitian matrices, which can be found in Sections 4.2 and 4.3 of \cite{HJ85}. For the readers convenience we include the results here. Given a Hermitian matrix $H\in\BC^{k \times k}$ we order the eigenvalues in increasing order, $\la_1(H)\leq \la_2(H)\leq \cdots \leq\la_k(H)$, multiplicities taken into account. For simplicity of the statement of our results, we also define $\la_j(H)=-\infty$ for $j<1$ and $\la_j(H)=\infty$ for $j>k$.

\begin{theorem}[Weyl's Inequality, Theorem 4.3.7 in \cite{HJ85}]\label{T:Weyl}
Let $M,N\in\BC^{m\times m}$ be Hermitian. Then for all integers $j,k>0$ we have
\begin{align}\label{WI}
  \la_{j+k-m}(M+N) & \leq \la_{j}(M)+\la_k(N)\leq \la_{j+k-1}(M+N).
\end{align}
\end{theorem}

When the number of positive and negative eigenvalues of the perturbation $N$ are known, we have the following result.

\begin{corollary}\label{C:Weyl}
Let $M,N\in\BC^{m\times m}$ be Hermitian. Assume $r_-$ and $r_+$ are the number of negative and positive eigenvalues of $N$, multiplicities taken into account. Then for any integer $j\geq 0$ we have
\begin{align*}
&\la_{j-r_+}(M+N) \leq \la_j(M)\leq \la_{j+r_- }(M+N), \\
&\quad\la_{j-r_-}(M) \leq \la_j(M+N)\leq \la_{j+r_+ }(M).
\end{align*}
\end{corollary}

\begin{proof}[\bf Proof]
If $r_+=m$ or $r_-=m$, then $N$ is positive definite or negative definite, respectively, and the validity of the claim follows from Theorem 4 in \cite{S92}. Hence assume $r_+\neq m$ and $r_-\neq m$. For $j=0$ it is easily verified that the inequalities hold. Let $j> 0$. Note that $\la_{k}(N)\leq 0$ when $k\leq m-r_+$. Therefore, using \eqref{WI} with $k=m-r_+$, we have
\[
\la_{j-r_+}(M+N)=\la_{j-(m-r_+)-m}(M+N) \leq \la_j(M) +\la_{m-r_+}(N)\leq \la_j(M).
\]
Moreover, we have $\la_k(N)\geq 0$ for $k\geq r_-+1$, so that \eqref{WI} with $k=r_-+1$ yields
\[
\la_{j+r_-}(M+N)=\la_{j+(r_-+1)-1}(M+N)\geq \la_{j}(M) +\la_{r_+ +1}(N)\geq \la_{j}(M).
\]
This proves the first pair of inequalities. For the second set of inequalities, apply the same argument with $M$ and $N$ replaced by $M+N$ and $-N$, respectively, noting that $-N$ has $r_+$ negative eigenvalues and $r_-$ positive eigenvalues, multiplicities taken into account.
\end{proof}

\begin{theorem}[Cauchy Interlacing Theorem, Theorem 4.3.15 in \cite{HJ85}]\label{T:Cauchy}
For a Hermitian $H\in\BC^{(m+k) \times (m+k)}$, partitioned accordingly as
\begin{equation}\label{Hdec}
H=\mat{cc}{M& K\\ K^*& N},
\end{equation}
we have
\begin{equation}\label{HM-EigIneq}
\la_j(H) \leq \la_j(M)\leq \la_{j+k}(H),\quad j\geq 0.
\end{equation}
\end{theorem}

Using the above results, we can prove the following result for the poles and zeros of functions in $\PRO_m$.

\begin{theorem}\label{T:Interlace}
Let $F\in\PRO_m$ be given by a minimal state space realization \eqref{RealFormPRO2}-\eqref{RealFormPROcon2}, so that $F^{-1}$ has a minimal state space realization as in Theorem \ref{T:inv2}. Then for any integer $j\geq 0$ we have
\begin{equation}\label{EigIneq}
\begin{aligned}
&\la_{j-\frac{m_2}{2}-m_3}(iA_\tu{inv})\leq \la_j(iA) \leq \la_{j+1}(iA)\leq \la_{j+\frac{m_2}{2}+m_1+1}(iA_\tu{inv}), \\
&\la_{j-\frac{m_2}{2}-m_1}(iA)\leq \la_j(iA_\tu{inv}) \leq \la_{j+1}(iA_\tu{inv})\leq \la_{j+\frac{m_2}{2}+m_3+1}(iA).
\end{aligned}
\end{equation}
In particular, if $0\leq \om_j<\om_{j+1}$ are such that $i \om_j$ and $i\om_{j+1}$ are subsequent poles of $F$, then in the interval $(i \om_j,i\om_{j+1})$ on $i\BR$ $F$ can have zeros whose multiplicities do not add up to more than $m$. Moreover, if $0\leq \nu_j<\nu_{j+1}$ are such that $i \nu_j$ and $i\nu_{j+1}$ are subsequent zeros of $F$, then in the interval $(i \nu_j,i\nu_{j+1})$ on $i\BR$ $F$ can have poles whose multiplicities do not add up to more than $m$.
\end{theorem}

We should remark here, that, unlike in the scalar case, for $m>1$ it is possible that poles and zeros of $F\in\PRO_m$ occur at the same point on $i\BR$. Hence, as in the theorem, if $i \om_j$ and $i\om_{j+1}$ are subsequent poles of $F$, then zeros with a multiplicities adding up to at most $m$ can occur between $i \om_j$ and $i\om_{j+1}$, but the theorem does not exclude the possibility that $F$ also has zeros at $i \om_j$ and $i\om_{j+1}$.

\begin{proof}[\bf Proof of Theorem \ref{T:Interlace}]
Let $A$, $\what{A}$, $\wtil{A}$ and $A_\tu{inv}$ be as in steps (i)-(iii) above. Then  $iA$, $i\what{A}$, $i\wtil{A}$ and $iA_\tu{inv}$ are Hermitian, hence with eigenvalues on $\BR$ which are mirrored in 0 because the matrices $A$, $\what{A}$, $\wtil{A}$ and $A_\tu{inv}$ are real skew-symmetric. Also, the perturbation $\La:=-B_2 D_{22}^{-1}B_2^T$ in step (i) is real skew-symmetric and has a rank of at most $m_2$ so that $i\La$ has at most $m_2/2$ positive eigenvalues and at most $m_2/2$ negative eigenvalues. Therefore, by Corollary \ref{C:Weyl} we have
\[
\la_{j-\frac{m_2}{2}}(i\what{A})\leq \la_j(iA)\ands
\la_{j+1}(iA)\leq \la_{j+\frac{m_2}{2}+1}(i\what{A}).
\]
Since $\widetilde{A}=\sbm{\what{A} & \star \\ \star & \star}\in\BR^{(n+m_1)\times (n+ m_1)}$, we can apply Theorem \ref{T:Cauchy} to obtain
\[
\la_{j-\frac{m_2}{2}}(i\wtil{A}) \leq \la_{j-\frac{m_2}{2}}(i\what{A}) \ands
\la_{j+\frac{m_2}{2}+1}(i\what{A})\leq \la_{j+\frac{m_2}{2}+m_1+1}(i\wtil{A}).
\]
Furthermore, after a change of basis, we can write $\widetilde{A}=\sbm{A_\tu{inv} & \star \\ \star & \star}$ with $\widetilde{A}$ of size $(n+m_1)\times (n+ m_1)$ and $A_\tu{inv}$ of size $(n+m_1-m_3)\times (n+ m_1-m_3)$. Hence, again applying  Theorem \ref{T:Cauchy} we obtain
\[
\la_{j-\frac{m_2}{2}-m_3}(iA_\tu{inv})\leq \la_{j-\frac{m_2}{2}}(i\wtil{A}) \ands
\la_{j+\frac{m_2}{2}+m_1+1}(i\wtil{A})\leq \la_{j+\frac{m_2}{2}+m_1+1}(iA_\tu{inv}).
\]
Putting these inequalities together we find that
\[
\la_{j-\frac{m_2}{2}-m_3}(iA_\tu{inv})\leq \la_{j-\frac{m_2}{2}}(i\wtil{A}) \leq \la_{j-\frac{m_2}{2}}(i\what{A})\leq \la_j(iA)
\]
and
\[
\la_{j+1}(iA)\leq \la_{j+\frac{m_2}{2}+1}(i\what{A})\leq \la_{j+\frac{m_2}{2}+m_1+1}(i\wtil{A})\leq \la_{j+\frac{m_2}{2}+m_1+1}(iA_\tu{inv}).
\]
Hence we proved the first set of inequalities in \eqref{EigIneq}. The second set of inequalities in \eqref{EigIneq} follows by a similar analysis, reversing the construction from $A$ to $A_\tu{inv}$.
\end{proof}

Note that it may happen that the perturbation $\La=-B^2 D_{22}^{-1}B_2^T$ has rank $2d< m_2$. In this case, the proof shows that the inequalities in \eqref{EigIneq} can be improved by replacing $m_2/2$ by $d$.

We conclude this paper with an example illustrating the pole-zero properties of $\PRO_m$ functions.

\begin{example}\label{E:pole-zero}
Consider $F\in \PRO_2$ given in the state space realization form of Theorem \ref{T:Transfer} with
\begin{align*}
    A &= \diag \left(\begin{bmatrix}0 & 1 \\ -1 & 0 \end{bmatrix}, \begin{bmatrix}0 & 2 \\ -2 & 0 \end{bmatrix},\begin{bmatrix}0 & 3 \\ -3 & 0 \end{bmatrix},\begin{bmatrix}0 & 4 \\ -4 & 0 \end{bmatrix},\begin{bmatrix}0 & 1 \\ -1 & 0 \end{bmatrix},\begin{bmatrix}0 & 5 \\ -5 & 0 \end{bmatrix} \right),\\
    B^T&=\sbm{0 & \frac{1}{10000} & \frac{1}{10} & \frac{5}{1000} & 0 & 0 &\frac{5}{1000} & 0 & 0 & 0 & 0 &\frac{1}{1000} \\ 1000 & 0 & 1 & 0 & 0 & \frac{1}{10000} & 0 & 0 & 0 & \frac{1}{1000} & 0 & 0 }, \quad D=\begin{bmatrix}0 & 50 \\ -50 & 0 \end{bmatrix}, M=0.
\end{align*}
 Then $m=m_2=2$ and $m_1=m_3=0$. Hence $F$ has no pole and no zero at $\infty$. One can verify that
 \begin{equation*}
    \rank \begin{bmatrix}B & AB & A^2B & \hdots & A^{11}B \end{bmatrix} =12.
\end{equation*}
Hence $(A,B)$ is a controllable pair. In particular, the state space realization in \eqref{RealFormPRO} is minimal, so that the (finite) poles of $F$ coincide with the eigenvalues of $A$:
\[
\pm 5i,\ \pm 4i,\ \pm 3i,\ \pm 2i,\ \pm 1i \mbox{ (multiplicity 2)}.
\]
In this case, since $m_1=m_3=0$, the state matrix of $F^{-1}$ is given by
\[
A_\tu{inv}=A-B D^{-1}B^T,
\]
a rank 2 perturbation of $A$. Using Matlab we found the eigenvalues of $A_\tu{inv}$ to be
\begin{align*}
-\lambda_{1}\left(A_{\textup{inv}} \right) = \lambda_{12}\left(A_{\textup{inv}}\right) = 5.000052 i,\quad
-\lambda_{2}\left(A_{\textup{inv}} \right) = \lambda_{11}\left(A_{\textup{inv}}\right) = 4.002068 i, \\
-\lambda_{3}\left(A_{\textup{inv}} \right) = \lambda_{10}\left(A_{\textup{inv}}\right) = 3.00000000000012 i,\
-\lambda_{4}\left(A_{\textup{inv}} \right) = \lambda_{9}\left(A_{\textup{inv}}\right) = 2.921053 i, \\
-\lambda_{5}\left(A_{\textup{inv}} \right) = \lambda_{8}\left(A_{\textup{inv}}\right) = 1 i,\quad
-\lambda_{6}\left(A_{\textup{inv}} \right) = \lambda_{7}\left(A_{\textup{inv}}\right) = 0.682921 i,
\end{align*}
which correspond to the zeros of $F$. It follows that there is one zero below $-5i$, one in each of the intervals $(-5i,-4i)$, $(-4i,-3i)$ and $(-3i,-2i)$ and two in the interval $(-i,i)$, with $\pm i$ the only points on $i\BR$ where both a pole and a zero coexist (although $\pm 3 i$ may have been missed as a zero by a round off error). The example shows, in particular, that it may occur that between two subsequent poles, there are zeros with multiplicities that add up to $m=2$, while these two poles are also zeros of $F$.
\end{example}

\paragraph{\bf Acknowledgments}
This work is based on research supported in part by the National Research Foundation of South Africa (NRF) and the DSI-NRF Centre of Excellence in Mathematical and Statistical Sciences (CoE-MaSS). Any opinion, finding and conclusion or recommendation expressed in this material is that of the authors and the NRF and CoE-MaSS do not accept any liability in this regard.

\end{document}